\documentclass[12pt,a4paper]{amsart}

\usepackage{amsfonts}
\usepackage{amsmath}
\usepackage{amssymb}
\usepackage{amsthm}
\usepackage{mathrsfs}
\usepackage{color}
\usepackage{bbm}

\usepackage[normalem]{ulem}

\usepackage{graphicx}
\usepackage[all]{xy}

\usepackage{enumerate}

\usepackage[colorlinks]{hyperref}

\newcommand{\f}{\varphi}

\newcommand{\aA}{{\mathcal{A}}}
\newcommand{\bB}{{\mathcal{B}}}
\newcommand{\cC}{{\mathcal{C}}}
\newcommand{\dD}{{\mathcal{D}}}

\newcommand{\fF}{{\mathcal{F}}}

\newcommand{\oO}{{\mathcal{O}}}

\newcommand{\qQ}{{\mathcal{Q}}}
\newcommand{\rR}{{\mathcal{R}}}

\newcommand{\la}{\lambda}

\newcommand{\kk}{\mathbbm{k}}

\newcommand{\lle}{\mbox{\raisebox{0.25ex}{${\scriptscriptstyle\le}$}}}
\newcommand{\gge}{\mbox{\raisebox{0.25ex}{${\scriptscriptstyle\ge}$}}}

\newcommand{\tr}{{\mbox{$t$-struc}\-tu\-r}}

\newcommand{\bcdot}{{\mbox{\boldmath{$\cdot$}}}}

\newcommand{\Hom}{\mathop{\textrm{Hom}}\nolimits}
\newcommand{\Ext}{\mathop{\operatorname{Ext}}\nolimits}
\newcommand{\Dec}{\mathop{\operatorname{Dec}}\nolimits}
\newcommand{\End}{\mathop{\textrm{End}}\nolimits}

\newcommand{\Coh}{\mathop{\operatorname{Coh}}\nolimits}

\newcommand{\Conn}{\mathop{\textrm{Conn}}\nolimits}

\newcommand{\opp}{\mathop{\textrm{op}}\nolimits}

\newcommand{\modu}{\mathop{\textrm{mod--}}}

\newcommand{\wt}[1]{\widetilde{#1}}

\newcommand{\ol}[1]{\overline{#1}}

\newtheorem{LEM}{Lemma}[section]

\newtheorem*{THM*}{Theorem}
\newtheorem{THM}[LEM]{Theorem}
\newtheorem{PROP}[LEM]{Proposition}
\newtheorem{COR}[LEM]{Corollary}

\theoremstyle{definition}
\newtheorem{EXM}[LEM]{Example}
\newtheorem{REM}[LEM]{Remark}

\begin{document}
\title{Highest weight categories via pairs of dual exceptional sequences}

\today
	\author{Agnieszka Bodzenta}
\author{Alexey Bondal}
\date{\today}

\address{Agnieszka Bodzenta\\
	Institute of Mathematics, 
	University of Warsaw \\ Banacha 2 \\ Warsaw 02-097,
	Poland} \email{A.Bodzenta@mimuw.edu.pl}

\address{Alexey Bondal\\
	Steklov Mathematical Institute of Russian Academy of Sciences, Moscow, Russia, and \\
	Centre for Pure Mathematics, Moscow Institute of Physics and Technology, Russia, and\\
	Kavli Institute for the Physics and Mathematics of the Universe (WPI), The University of Tokyo, Kashiwa, Chiba 277-8583, Japan} 
\email{bondal@mi-ras.ru}

	\begin{abstract}
In this paper we present criteria in terms of dual pairs of exceptional sequences for an abelian category to be highest weight. The criteria are applied in three situations of geometric origin. We give new proofs for the facts that the category of perverse sheaves of middle perversity on complex-analytic manifolds with suitable conditions on the stratification is highest weight and that the derived coherent category of any Grassmannian has a t-structure with highest weight heart. Also we show that the abelian null category of any proper birational morphism of regular surfaces is highest weight. For this null category, we give a geometric description of some special objects related to the highest weight structure, such as standard, costandard and characteristic tilting objects. 
\end{abstract}
	\maketitle
	\tableofcontents
	
	\section{Introduction}
	
	Highest weight abelian categories, whose axiomatics was introduced by Cline, Parshall and Scott in \cite{CPS} are omnipresent in the Representation Theory ({\it cf.} \cite{ParSco}, \cite{DR}, \cite{Donkin1},\cite{Rou4},\cite{BrunStrop}, to mention few). In order to define a highest weight structure on an abelian category one needs to fix a (partial) order on the set $\Lambda$ of isomorphism classes of its irreducible objects. The category is highest weight if it possesses a set of so-called standard objects (defined with respect to the order on $\Lambda$) which satisfy some good conditions (see Section \ref{ssec_hwc} below).  
	
	Recently, several authors have given various new characterizations of highest weight categories. 
	In \cite{CirWoo}, Cipriani and Woolf obtained a criterion for an abelian category to be highest weight in terms of existence of a suitable stability condition on the category and a purely numerical criterion in terms of the decomposition at the level of the Grothendieck group of projective objects via standard objects. 
	
	In \cite{ACKT}, Adachi, Chan, Kimura and Tsukamoto describe highest weight structures on the category of modules over a given algebra in terms of the properties of characteristic tilting module, the central object of Ringel duality on higest weight categories. The authors of \cite{ACKT} put a consistency condition between the tilting module and simple modules over the algebra which imply that the module is characteristic tilting  for the chosen order on $\Lambda$.

	In \cite{BodBon4}, the authors of the present paper characterized highest weight categories as the left/right abelian envelopes of thin exact categories. Thin exact categories are exact categories together with a {\it thin filtration}, {\it i.e.} an admissible filtration (see \cite{BodBon4}) whose graded factors are categories of vector spaces. 
	From this perspective, thin exact categories make sense as the basic object of study in the theory of highest weight categories. The thin exact category inside the highest weight category is the full subcategory of standardly (or costandardly) filtered objects, which was originally considered by Claus Ringel  in \cite{Rin}.

	One can regard thin filtration as an exact category counterpart of the notion of exceptional sequence in triangulated category \cite{B}. Indeed, there is a clear direct connection between the two notions: the irreducible objects (with the prescribed order) in a thin category make an  exceptional sequence in its derived category, that of standard objects. 
	
	Note that the derived categories of a thin category and of the highest weight category which is its right (or left) abelian envelope, are equivalent. Thus we may consider the highest weight category as intermediate between the thin subcategory and the ambient derived category.
	It is natural to ask how to characterize highest weight categories in terms of the exceptional sequence in its derived category. The subject of this paper is to give such characterizations and to apply them to deduce that certain categories of geometric origin are highest weight.

	In the spirit of this question, Q. Liu and D. Yang asked in \cite{LiuYang}  whether an abelian category $\aA$ is highest weight if and only if $\dD^b(\aA)$ has a full exceptional sequence of objects in $\aA$. In \cite[Proposition 5.6]{Kalck} M. Kalck constructed a counterexample. We discuss it in more detail in Section \ref{ssec_Kalck}. 
	
	Recall that there is another exceptional sequence, that of costandard objects, in any highest weight category $\aA$. In \cite{BodBon4}, this sequence got the interpretation of irreducible objects in another thin exact subcategory of $\aA$, whose {\it left} abelian envelope is equal to $\aA$. It is well-known that, under the chosen order on $\Lambda$, this is also an exceptional sequence in $\dD^b(\aA)$, actually the left dual (see \cite{B} and Section \ref{ssec_dual_exc_coll}) to the sequence of standard objects ( {\it cf.} \cite[Proposition 4.19]{Rou4}).
	
	In this paper we give various characterization of highest weight structures on abelian categories in terms of the both exceptional sequences. It would be interesting to check how these results extend to (semi)infinite highest weight categories \cite{BrSt24}.
	
	In what concerns the recovering of the highest weight category $\aA$ from its derived category, it is essentially due to \cite{CPS} that $\aA$ is the heart of the \tr e glued along the filtration  induced by the exceptional sequence of co/standard objects in the derived category (see also Corollary \ref{cor_hwc_is_glued}).
	
	Our first criterion is for the heart of the \tr e glued along an exceptional sequence to be a highest weight category. Remarkably, the criterion provides the necessary and sufficient conditions purely in terms of vanishing of negative {\it Ext}-groups between the objects in the exceptional sequences. 

\begin{THM}(Theorem \ref{thm_univ_ext_tilt_gen})
	Let $\aA$ be the heart of the \tr e on a triangulated category $\dD$ glued along a full exceptional sequence $(E_1,\ldots,E_n)$  with left dual exceptional sequence $(F_n,\ldots, F_1)$. Then 
	\begin{enumerate}
		\item$\aA$ is a highest weight category with standard objects $\{E_i\}$ and costandard objects $\{F_i\}$ if and only if
		\begin{itemize}
		\item
		$\Hom_{\dD}(\bigoplus_{i=1}^n E_i, \bigoplus_{i=1}^n E_i[l])=0$, for all $l<0$.
		\item
		$\Hom_{\dD}(\bigoplus_{i=1}^n F_i, \bigoplus_{i=1}^n F_i[l])=0$, for all $l<0$. 
		\end{itemize}
		\item If $\dD$ is DG enhanced \cite{BK} and conditions of (1) are satisfied, then $\dD\cong \dD^b(\aA)$.
	\end{enumerate}
\end{THM}
	
	We give an interpretation of Ringel duality for highest weight categories in terms of duality of exceptional sequences in Section \ref{Ring-dual}. Theorem \ref{thm_Ringel} claims that if a highest weight category is glued along an exceptional  sequence, then the category glued along the left dual sequence is highest weight too, and it is the Ringel dual category.
	
	Our second criterion for an abelian category to be highest weight requires that both dual exceptional sequences are in the abelian category, thus strengthening the condition of Liu and Yang. Namely, we prove
	
	\begin{THM}(Theorem \ref{thm_main})\label{thm_A}
		A highest weight structure on an abelian category $\aA$ is uniquely defined by a pair  $(\{ E_i \},  \{ F_i \})$ of finite sequences of objects in $\aA$ such that
		 $(\{ E_i \},  \{ F_i \})$ is a dual pair of full exceptional sequences in $\dD^b(\aA)$.
		 The sets $\{ E_i \}$ and $\{ F_i \}$ are respectively the sets of standard and costandard objects of the corresponding highest weight category.
		
	\end{THM}


Theorem \ref{thm_A} is a corollary of the following more general statement where we study the restriction of a \tr e to the full triangulated subcategory $\langle E_1,\ldots,E_n\rangle$ generated by an exceptional sequence $(E_1,\ldots, E_n)$.
\begin{THM}(Theorem \ref{thm_general})\label{thm_intro_gen}
Let $\aA$ be the heart of a \tr e $(\dD^{\leq 0}, \dD^{\geq 0})$ on a triangulated category $\dD$ and $(E_1,\ldots, E_n)$ an exceptional sequence in $\dD$ with left dual $(F_n,\ldots, F_1)$. 
	\begin{enumerate}
		\item If $E_i\in \dD^{\leq 0}$ and $F_i \in \dD^{\geq 0}$ for all $i$, then the \tr e on $\dD$ restricts to the  	\tr e on $\cC:= \langle E_1,\ldots,E_n\rangle$ glued along the exceptional sequence $(E_1,\ldots, E_n)$. The heart of this \tr e is $\aA \cap \cC$. 
		\item If $E_i, F_i\in \aA$, for all $i$, then $\aA\cap \cC$ is a highest weight category with standard objects $\{E_i\}$ and costandard $\{F_i\}$.
	\end{enumerate}
\end{THM} 

Theorem \ref{thm_A} is the case of Theorem \ref{thm_intro_gen} when $\dD = \dD^b(\aA)$ and the exceptional sequence is full.

In Section \ref{ssec_dir_alg}, we illustrate how this correspondence between highest weight structures and pairs of dual exceptional sequences works on the example of the path algebra $A$ of any directed quiver with relations. There are at least two highest weight structures on the category of $A$-modules: the standard order of vertices in the directed quiver corresponds to the pair of exceptional sequences of simple and of indecomposable injective modules, while the inverse order of vertices corresponds to sequences of indecomposable projective and of simple modules.

In the rest of the paper, we concentrate on geometric applications of our technique. First, we give a short proof of the fact that the category of perverse sheaves on a complex analytic space stratified by consecutive divisors with contractible open parts and endowed with the middle perversity is highest weight \cite{MirVil, ParSco} (see Section \ref{perverses}). Then, we give an alternative proof (Theorem \ref{thm_Gr}) that the derived category of coherent sheaves on a Grassmannian is equivalent to the derived category of a highest weight category \cite{BuchLeuVdB} , \cite{Efi} (see Section \ref{secGrass}). 

Finally, we apply Theorem \ref{thm_intro_gen} to proper birational morphisms $f\colon X\to Y$ of regular surfaces. We consider the null-category $\cC_f := \{E\in \dD^b(\Coh(X))\,|\,Rf_* E=0\}$ and its subcategory, the abelian null-category $\mathscr{A}_f := \cC_f \cap \Coh(X)$. We prove that it is highest weight. 

In order to describe standard, costandard and characteristic tilting objects in $\mathscr{A}_f$, we consider the the partially ordered set $\textrm{Conn}(f)$ of decompositions $f = h \circ g$ of $f$ into a pair $(g,h)$  of proper birational morphisms of regular surfaces such that the exceptional divisor of $g$ is connected. We put $(g_1,h_1)\preceq (g_2, h_2)$ if $g_2$ factors via $g_1$.

Let $R_g\subset X$ denote the fundamental cycle of a birational morphism $g\in \textrm{Conn}(f) $, and $D_g\subset X$ its discrepancy divisor.

\begin{THM}(Theorems \ref{thm_surfaces} and \ref{thm_tilting_for_surf})\label{thm_intro_surfaces}
	The standard \tr e on $\dD^b(\Coh(X))$ restricts to a \tr e on $\cC_f$ with heart $\mathscr{A}_f $.  Category $\mathscr{A}_f$ is highest weight with standard objects $\{\oO_{R_g}(R_g)\}_{g\in \Conn(f)}$, costandard objects $\{\oO_{R_g}(D_g)\}_{g\in \Conn(f)}$ and the characteristic tilting object  $\bigoplus_{g\in \Conn(f)} \omega_X|_{D_g}$.
	Moreover, $\cC_f\cong \dD^b(\mathscr{A}_f)$. 
\end{THM}


We show on the examples of the minimal resolutions $f: X\to Y$  of type  $A_1$ and $A_2$ surface singularities that if one omits the condition for surface $Y$ to be regular, then neither the null-category $\mathscr{A}_f$ is necessarily a highest weight category, nor $\cC_f$ is necessarily equivalent to $\dD^b(\mathscr{A}_f)$ (see Examples \ref{typeA2} and \ref{typeA1}).

\vspace{0.3cm}
\noindent
\textbf{Acknowledgements.} We are thankful to A. Efimov, A. Fonarev and A. Nordskova for useful discussions. 
 The first author was partially supported by Polish National Science Centre grant 2021/41/B/ST1/03741.  
The work of A. Bondal was performed at the Steklov International Mathematical Center and supported by the Ministry of Science and Higher Education of the Russian Federation (agreement no. 075-15-2025-303). This work was supported by JSPS KAKENHI Grant Number JP23K20205 and by Kavli IPMU (WPI, MEXT), Japan.

\vspace{0.3cm}
\noindent
\textbf{Notations.}
We work over an algebraically closed base field $\kk$. 

For objects $Y,Z$ in a $\kk$-linear triangulated category $\dD$, we denote by $\Hom_{\dD}^{\bcdot}(Y,Z)$ the graded $\kk$-vector space $\bigoplus_{l\in \mathbb{Z}}\Hom_{\dD}(Y,Z[l])$.

For a noetherian $\kk$-scheme $X$ we denote by $\dD^b(X)$ the bounded derived category of coherent sheaves on $X$. Similarly, $\dD^b(\kk)$ is the bounded derived category of finite dimensional $\kk$-vector spaces.

\section{Preliminaries}

We recall the definition of highest weight categories and their relation to standarizable sequences in abelian categories. We further discuss dual pairs of exceptional sequences and the aisles of the \tr e glued along an exceptional sequence. 

\vspace{0.3cm}
\subsection{Highest weight categories}\label{ssec_hwc}~\\

 We say that a $\kk$-linear abelian category is \emph{of finite length} if it is $\Hom$ and $\Ext^1$-finite, contains finitely many non-isomorphic irreducible objects, and any object admits a finite filtration with irreducible graded components. We say that an abelian category of finite length is a \emph{Deligne finite category} if it has a projective generator. Deligne finite categories are exactly the categories $\modu A$ of finite dimensional right modules over finite dimensional algebras. Given a projective generator $P \in \aA$, one can take $A= \End_{\aA}(P)$, then the equivalence $\aA \xrightarrow{\simeq} \modu A$ is given by the functor $\Hom_{\aA}(P, -)$.

Let $\aA$ be a Deligne finite $\kk$-linear category, and $\Lambda$ a partial order on the set of  isomorphism classes of irreducible objects in $\aA$. Let $L(\la)$ denote the irreducible object corresponding to $\la \in \Lambda$, $P(\la)$ its projective cover, and $I(\la)$ its injective hull. We assume that the order is \emph{adapted}, i.e. for any $M \in \aA$ with the top $L(\la_1)$, the socle $L(\la_2)$, and $\la_1$ and $\la_2$ not comparable in $\Lambda$, there exists $\mu $ such that  $\mu\succeq \la_1$,  $\mu \succeq \la_2$, and $L(\mu)$ is a graded factor of a filtration on $M$. We say that $(\aA, \Lambda)$ is a  \emph{highest weight category} \cite{CPS} if $\aA$ has a set, called the set of \emph{standard objects}, $\Delta(\la)$, labelled by $\la \in \Lambda$, such that 
\begin{itemize}
	\item[(st1)] there exists an epimorphism $\Delta(\la) \to L(\la)$
	whose kernel admits a filtration with graded factors $L(\mu)$, where $\mu \prec \la$,
	\item[(st2)] there exists an epimorphism $P(\la) \to \Delta(\la)$
	whose kernel admits a filtration with graded factors $\Delta(\mu)$, where $\mu \succ\la$.
\end{itemize} 

It follows from 
\cite[Theorem 4.3]{ParSco} that $\aA^{\opp}\cong (\textrm{mod-}A)^{\opp} \cong \textrm{mod-}A^{\opp}$ is also a highest weight category with the same partial order $\Lambda$ on the set of simple objects $\{L(\la)^{\opp}\}$. By transporting the standard objects from $\aA^{\opp}$ to $\aA$, we get \emph{costandard objects} $\nabla(\la)$, labelled by $\la \in \Lambda$, defined by conditions: 
\begin{itemize}
	\item[(cost1)] there exists an injective morphism $L(\la) \to \nabla(\la)$ whose cokernel admits a filtration with graded factors $L(\mu)$, where $\mu \prec \la$,
	\item[(cost2)] 
	there exists an injective morphism $\nabla(\la) \to I(\la)$ whose cokernel admits a filtration with graded factors $\nabla(\mu)$, where  $\mu \succ\la$.
\end{itemize} 

A highest weight category $\aA$ is equivalent to the category $\modu A$ where $A$ is a \emph{quasi-hereditary} algebra and \textit{vice versa} \cite{CPS}.

Consider an abelian category $\aA$. We say that highest weight structures $(\aA, \Lambda)$, $(\aA, \Lambda')$ are \emph{equivalent} if they have the same sets of isomorphism classes of standard objects $\{\Delta(\la)\}_{\la \in \Lambda} = \{\Delta(\la')\}_{\la' \in \Lambda'}$. 

Similarly, two exceptional sequences in a triangulated category are \emph{equivalent} if the sets of isomorphism classes of their objects are equal (but the full order of the sequence might change).

\vspace{0.3cm}
\subsection{Standarizable sequences}~\\

Given a set of objects $\{E_i\}	$ in an abelian category $\aA$ we denote by $\fF(\{E_i\})\subset \aA$ the extension closure of $E_i$'s, i.e. the full subcategory of objects $X\in \aA$ that admit a finite filtration $0 =X_0 \subset X_1\subset \ldots \subset X_N = X$ with graded factors $X_l/X_{l-1}$ isomorphic to one of the $E_i$'s.

Consider a sequence $\epsilon = (E_1,\ldots,E_n)$ of objects in a $\kk$-linear abelian category $\aA$ such that $\dim_{\kk}\Ext^1_{\aA}(E_i,E_j)$ are finite  for any pair $i,j \in \{1,\ldots, n\}$. The \emph{iterated universal extension} $P $ of  $\epsilon$ is defined as a sum $P = \bigoplus_{i=1}^nP_i$ recursively on the length $n$ of the sequence. For $n=1$, we put $P_1 = E_1$. For $n>1$, let $\bigoplus_{i=1}^{n-1}Q_i$ be the iterated universal extension of $(E_1,\ldots, E_{n-1})$. Then, for $1\leq i < n$, $P_i$ is the universal extension of $Q_i$ by $E_n$, and $P_n := E_n$.

Recall that given a pair $(Q,T)$ of objects in a $\kk$-linear abelian category $\cC$ such that $\dim_{\kk}\Ext_{\cC}^1(Q,T)$ is finite, the \emph{universal extension} of $Q$ by $T$ is the extension of $Q$ by $T\otimes \Ext^1_{\cC}(Q,T)^\vee$:
\begin{align}\label{univ_extens}
0 \to T\otimes \Ext^1_{\cC}(Q, T)^\vee \to R \xrightarrow{\pi} Q \to 0
\end{align}  
given by the canonical element in $\Ext_{\cC}^1(Q, T \otimes \Ext_{\cC}^1(Q,T)^\vee) \cong \Ext_{\cC}^1(Q, T) \otimes \Ext_{\cC}^1(Q, T)^\vee\simeq \End(\Ext^1_{\cC}(Q,T))$. 

The notion of an iterated universal extension of $\epsilon$ plays an important role when $\epsilon$ is a \emph{standarizable sequence}, i.e. when $ \dim_{\kk}\Hom_{\aA}(E_i,E_j)$ are finite for any pair $i,j\in \{1,\ldots, n\}$ and $\textrm{rad}(E_i,E_j) =0 = \Ext^1(E_i,E_j)$ for any pair $i\geq j$, see \cite{DR}.

\begin{THM}\cite[Theorem 2]{DR}\label{thm_DR}
	Let $\epsilon = (E_1,\ldots, E_n)$ be a standarizable sequence in a $\kk$-linear abelian category $\aA$ and $P$ the iterated universal extension of $\epsilon$. Then the category $\wt{\aA} = \textrm{mod-}\End_{\aA}(P)$ with the full order $1\leq \ldots \leq n$ on the set of irreducible objects in $\wt{\aA}$ is a highest weight category with standard objects $\Delta_i = \Hom_{\aA}(P, E_i)$. Moreover, the subcategory $\fF(\{E_i\})$ of $\aA$ and the subcategory $\fF(\{\Delta_i\})$ of $\wt{\aA}$ are equivalent.
\end{THM}

It is shown in \cite{BodBon4} that the abelian category $\wt{\aA}$ can be intrinsically reconstructed from the exact category $\fF(\{E_i\})$  as its right abelian envelope.

\vspace{0.3cm}
\subsection{Pairs of dual exceptional sequences in highest weight categories}\label{ssec_dual_exc_coll}~\\

 For an exceptional pair $(E, F )$ in a triangulated category $\dD$, the \emph{left and right mutation} are defined via \emph{canonical} exact triangles:
\begin{align*}
&\Hom_{\dD}^\bcdot(E,F) \otimes E \to F \to L_EF,& &R_FE \to E \to F \otimes \Hom_{\dD}^\bcdot(E,F)^\vee.&
\end{align*}
Note the shift with respect to the original definition in \cite{B}. 

Mutations define a braid group action on the set of exceptional sequences in a triangulated category $\dD$ \cite{B}. In particular, given an exceptional sequence $(E_1,\ldots, E_n)$, its \emph{left dual} exceptional sequence can be canonically defined as the sequence $(F_n,\ldots,F_1)$ with $F_i =L_{E_1}\ldots L_{E_{i-2}} L_{E_{i-1}} E_i$, for any $i\geq 2$, and $F_1 = E_1$. 

In the case the exceptional sequence is {\it full}, {\it i.e.} generates the triangulated category $\dD$, the left dual exceptional sequence is uniquely characterised by the Hom-duality property 
\begin{equation}\label{eqtn_dual_collec}
\Hom_{\dD}(E_i,F_j[l]) = \left\{\begin{array}{cl}\kk, & \textrm{if }l=0, i=j,\\0, & \textrm{otherwise.} \end{array} \right. 
\end{equation} 

The following fact is well-known (\emph{cf.} \cite[Proposition 4.19]{Rou4}):
\begin{LEM}\label{lem_two_exc_coll}
	Let $(\aA, \Lambda)$ be a highest weight category. Then, for any choice of a full order $\la_1 \leq \ldots \leq \la_n$ on $\Lambda$ compatible with the poset structure, the
	sequence $(\Delta(\lambda_1),\ldots, \Delta(\lambda_n) )$ is full exceptional in $\dD^b(\aA)$ with the left dual sequence $(\nabla(\lambda_n),\ldots, \nabla(\lambda_1) )$.
\end{LEM}
\begin{REM}\label{rem_hws_by_st}
	Consider a highest weight category $(\aA, \Lambda)$ and the full exceptional sequence $(\Delta(\lambda_1),\ldots, \Delta(\lambda_n) )$ as above.
	Then, for the full order  $\Lambda' = \{\la_1\leq \ldots \leq \la_n\}$ on the set 
	of irreducible objects in $\aA$, the highest weight structure $(\aA, \Lambda')$ is equivalent to the original one $(\aA, \Lambda)$, as it is for any refinement of the partial order on $\Lambda$.
\end{REM}
\vspace{0.3cm}
\subsection{$t$-structures glued along exceptional sequences}~\\

By a \emph{$t$-category} we mean  a pair of a triangulated category $\dD$ and a \tr e on it \cite{BBD}.

Let $(E_1,\ldots, E_n)$ be a full exceptional sequence  in a triangulated category $\dD$. We denote by $\dD_s :=\langle E_1,\ldots, E_s\rangle$ the full subcategory generated by the first $s$ objects of the sequence and consider the \tr e 
on $\dD$ glued via the admissible filtration
\begin{equation}\label{eqtn_filtration}
\dD_1 \subset \dD_2 \subset \ldots \subset \dD_{n-1}\subset \dD 
\end{equation}
from the standard \tr es on $\dD_i/\dD_{i-1}  \cong \dD^b(\kk)$ (see \cite[Theorem 1.7]{BodBon2}). We call this \tr e {\it glued along the exceptional sequence} $(E_1,\ldots, E_n)$. By \cite[Proposition 2]{Bez1}, the heart $\aA$ of the glued \tr e is a finite length abelian category. 

Let $( F_n,\ldots, F_1)$ be the exceptional  sequence left dual to $(E_1,\ldots, E_n)$. One can describe elements of the sequence and its dual by means of the functors in the  recollements 
\begin{equation}\label{eqtn_recol_l}
\xymatrix{\dD_{s-1} \ar[r]|(0.55){i_{s*}} & \dD_{s} \ar[r]|(0.45){j_s^*} \ar@<2ex>[l]|(0.45){i_s^!} \ar@<-2ex>[l]|(0.45){i_s^*} & \dD^b(k) \ar@<2ex>[l]|(0.55){{j_s}_*} \ar@<-2ex>[l]|(0.55){{j_s}_!}}
\end{equation} 
as $F_s = j_{s*}(\kk) $ and $E_s =j_{s!}(\kk) $.

\begin{LEM}\label{lem_glued_t-str}
	Let $(E_1, \ldots, E_n ) $ be a full exceptional sequence in a triangulated category $\dD$ and $( F_n, \ldots, F_1 )$ its left dual. The \tr e on $\dD$ glued along $(E_1, \ldots, E_n ) $
	is determined by the equalities:
	\begin{align*}
	\dD^{\lle 0} = \{D\in \dD\,|\, \Hom_{\dD}(D, F_s[l])=0, \textrm{ for all } l<0
	\},\\
	\dD^{\gge 0} = \{D\in \dD\,|\, \Hom_{\dD}(E_s, D[l])=0, \textrm{ for all } l<0
	\}.
	\end{align*}
\end{LEM}
\begin{proof}
	We proceed by induction on the length $n$ of the exceptional sequence. The case $n=1$ is clear.
	
	Functor $j_n^*$ in recollement (\ref{eqtn_recol_l}) for $s=n$ is right adjoint to the embedding of ${j_n}_!(\kk) = E_n$ and left adjoint to the embedding of ${j_n}_*(\kk) = F_n$. Hence, $j_n^*(-) \cong \Hom^\bcdot_{\dD}(-,F_n)^\vee \cong \Hom^\bcdot_{\dD}(E_n,-)$. Indeed, an inclusion $a_*\colon \dD^b(\kk) \to \dD$ of an exceptional object $E\in \dD$ admits the left adjoint $a^*(-) = \Hom^\bcdot_{\dD}(-,E)^\vee$ and the right adjoint $a^!(-) = \Hom^\bcdot_{\dD}(E,-)$. Then, the standard description of the glued \tr e (see \cite[Theorem 1.4.10]{BBD}) and the induction hypothesis give:
	\begin{align*}
	\dD^{\leq 0} &= \{D\in \dD\,|\, i_n^*(D)\in \dD_{s-1}^{\leq 0}, j_n^*(D) \in \dD^b(\kk)^{\leq 0} \} \\
	&= \{D\in \dD\,|\, \Hom_{\dD}^{\bcdot}(i_n^*(D),F_s)\in \dD^b(\kk)^{\geq 0}, \textrm{ for} 1\leq s<n, \Hom^\bcdot_{\dD}(D, F_n)^\vee  \in \dD^b(\kk)^{\leq 0} \} \\
	&= \{D\in \dD\,|\, \Hom_{\dD}^{\bcdot}(D, F_s)\in \dD^b(\kk)^{\geq 0} \textrm{ for}1\leq s\leq n\}\\
	\dD^{\geq 0} &= \{D\in \dD\,|\, i_n^!(D)\in\dD_{s-1}^{\geq 0}, j_n^*(D) \in \dD^b(\kk)^{\geq 0} \}\\
	&= \{D\in \dD\,|\, \Hom_{\dD}^{\bcdot}(E_s,i_n^!(D))\in \dD^b(\kk)^{\geq 0}, \textrm{ for } 1\leq s<n, \Hom^\bcdot_{\dD}(E_n,D)\in \dD^b(\kk)^{\geq 0} \}\\
	& = \{D\in \dD \,|\, \Hom_{\dD}^{\bcdot}(E_s, D)\in \dD^b(\kk)^{\geq 0} \textrm{ for }1\leq s\leq n\}.\qedhere
	\end{align*}
\end{proof}

\begin{COR}\label{cor_dual_heart}
	Let $\aA$ be the heart of the \tr e on $\dD$ glued along a full exceptional sequence.
	Then $\aA^{\opp}$ is equivalent to the heart of the \tr e on $\dD^{\opp}$ glued along the left dual exceptional sequence $(F_1, \ldots, F_n)$ (mind the order in the opposiite category!).
\end{COR}
\begin{proof}
	If $(E_1, \ldots, E_n)$ is the original exceptional sequence and $(F_n, \ldots, F_1)$ its left dual, then,  in $\dD^{\opp}$, the sequence $(E_n, \ldots, E_1)$ is left dual to $(F_1, \ldots, F_n)$. Lemma \ref{lem_glued_t-str} implies that, for the glued \tr es on $\dD$ and $\dD^{\opp}$, we have: $(\dD^{\opp})^{\leq 0} \simeq (\dD^{\geq 0})^{\opp}$ and $(\dD^{\opp})^{\geq 0} \simeq (\dD^{\leq 0})^{\opp}$, hence the corollary. 
\end{proof}

\section{Highest weight categories via exceptional sequences}\label{sec_hwc_via_ex_col}

We obtain equivalent conditions for the heart of a \tr e glued along an exceptional sequence to be a highest weight category. We apply this to interpret Ringel duality for highest weight categories in terms of duality for exceptional sequences.  

Then we consider an exceptional sequence in a \textit{t}-category and give conditions on the sequence and its left dual under which the \tr e  restricts to the subcategory generated by the sequence. Under further conditions  the heart of the restricted \tr e  is proved to be highest weight. We conclude with a bijection of equivalence classes of highest weight structures on an abelian category $\aA$ and dual pairs of full exceptional sequences in $\dD^b(\aA)$ both consisting of objects in $\aA$. 

\vspace{0.3cm}
\subsection{A criterion for the heart of the glued \tr e to be highest weight}~\\

First, we consider a full exceptional sequence in a triangulated category and give necessary and sufficient conditions for the heart of the \tr e glued along it to be a highest weight category in terms of the vanishing of negative Homs in the sequence and in its left dual one.

\begin{THM}\label{thm_univ_ext_tilt_gen}
	Let $\epsilon =(E_1,\ldots,E_n)$ be a full exceptional sequence in  $\dD$, $(F_n,\ldots,F_1)$ its left dual sequence, and $\aA$ the heart of the \tr e glued along $(E_1,\dots ,E_n)$.
	Then
	\begin{enumerate}
		\item $\aA$ has a structure of the highest weight category with standard objects $\{E_i\}$, costandard $\{F_i\}$, and the order on irreducible objects induced by $\epsilon$,
		if and only if 
		\begin{itemize}
		\item
		$\Hom_{\dD}(\bigoplus_{i=1}^n E_i, \bigoplus_{i=1}^n E_i[l])=0$, for all $l<0$.
		\item
		$\Hom_{\dD}(\bigoplus_{i=1}^n F_i, \bigoplus_{i=1}^n F_i[l])=0$, for all $l<0$. 
		\end{itemize}
		\item If $\dD$ is idempotent complete and DG enhanced \cite{BK} and conditions of (1) are satisfied, then $\dD\cong \dD^b(\aA)$.
	\end{enumerate}
\end{THM}
\begin{proof}
	If $E_i$ and $F_i$ are objects of the heart $\aA$, the vanishing  of the negative Hom-groups in $(1)$ follows.
	
	Assume that $\Hom_{\dD}(\bigoplus_{i=1}^n E_i, \bigoplus_{i=1}^n E_i[l]) = 0 = \Hom_{\dD}(\bigoplus_{i=1}^n F_i, \bigoplus_{i=1}^n F_i[l])$ for all $l<0$. By Lemma \ref{lem_glued_t-str} and the 
	Hom-duality given by \eqref{eqtn_dual_collec}, $E_i$'s and $F_j$'s are objects of the heart $\aA$. Denote by $P = \bigoplus_{i=1}^nP_i$ the iterated universal extension of the standarizable sequence $(E_1,\ldots,E_n)$ in $\aA$. 
	We prove by induction on the length $n$ of the exceptional sequence that $P$ is a strong titling generator in the triangulated category $\dD$ and $P$ is a projective generator in the abelian category $\aA$. The case $n=1$ is clear.
	
	Let us assume that the statement holds for any sequence of length less than $n$. Let $Q = \bigoplus_{i=1}^{n-1} Q_i$ be the iterated universal extension of $\{E_1,\ldots, E_{n-1}\}$. 
	
	Consider recollement \eqref{eqtn_recol_l} for $s=n$. Both $Q$ and $F_n = {j_n}_*(\kk)$ are objects of $\aA$ by Lemma \ref{lem_glued_t-str}, hence $\Hom_{\dD}(Q, F_n[l]) =0$, for all $l<0$. Moreover, by the definition of the glued \tr e \cite{BBD}, $i_n^*F_n$ is an object of the negative aisle of the glued \tr e on $\dD_{n-1}$. It follows that $i_n^!{j_n}_!(\kk) = i_n^*{j_n}^*(\kk)[-1] = i_n^*F_n[-1]$ is an object of $\dD_{n-1}^{\leq 1}$. Hence, $\Hom_{\dD_{n-1}}(Q, i_n^!{j_n}^!(\kk)[l]) \cong \Hom_{\dD}(Q, {j_n}_!(\kk)[l]) = \Hom_{\dD}(Q, E_n[l])=0$, for $l\geq 2$.
	
	As $\aA\subset \dD$ is the heart of a \tr e, $\Ext^1_{\aA}(Q_i E_n) \cong \Hom_{\dD}(Q_i, E_n[1])$. 
	Thus, for $i\leq n-1$, the short exact sequence defining $P_{i}$ gives an exact triangle
	\begin{align}\label{eqtn_def_P_i}
	 E_n \otimes \Hom_{\dD}(Q_i, E_n[1])^{\vee} \to P_i \to Q_i 
	\end{align}
	in $\dD$. First, we check that $P $ is tilting.
	
	Consider $i<n$.
	The map $\Hom_{\dD}(E_n \otimes \Hom_{\dD}(Q_i, E_n[1])^\vee, E_n) \to \Hom_{\dD}(Q_i, E_n[1])$ induced by \eqref{eqtn_def_P_i} is an isomorphism in view of the definition of the class in $\Ext_{\aA}^1(Q_i, E_n\otimes \Hom_{\dD}(Q_i, E_n[1])^{\vee})$ defining $P_i$. Moreover, $\Hom_{\dD}(E_n, E_n[l]) =0$,  for $l\neq 0$, as $E_n$ is exceptional, and $\Hom_{\dD}(Q_i, E_n[l]) = 0$, for $l\geq 2$, as observed above. It follows from the long exact sequence obtained by applying $\Hom_{\dD}(-, E_n)$ to \eqref{eqtn_def_P_i} that $\Hom_{\dD}(P_i, E_n[l])=0$, for $l\geq 1$. Since $P_i$ and $E_n=P_n$ are both objects of the heart $\aA$, we get vanishing of $\Hom_{\dD}(P_i, E_n[l]) = \Hom_{\dD}(P_i, P_n[l])$, for all $l \neq 0$ and all $i\in \{1, \ldots,n-1\}$. As $P_n=E_n$ is exceptional, we also get vanishing of  $\Hom_{\dD}(P_n, P_n[l])$, for every $l\neq 0$.
		
	Object $Q_j$ lies in $\dD_{n-1}$, hence $\Hom_{\dD}(P_n, Q_j[l]) = \Hom_{\dD}(E_n, Q_j[l])=0$, for any $l$.
	 In view of  \eqref{eqtn_def_P_i}, it follows that, if $i,j<n$, then $\Hom_{\dD}(P_i, Q_j[l]) \cong \Hom_{\dD}(Q_i, Q_j[l])$, which  vanish for $l \neq 0$ by the inductive assumption that $Q$ is tilting. 
	 We conclude that $\Hom_{\dD}(P_i, P_j[l]) =0$ for $l \neq 0$, $i\leq n$ and $j<n$. Indeed, by the long exact sequence obtained by applying $\Hom_{\dD}(P_i, -)$ to \eqref{eqtn_def_P_i} the statement follows from vanishing of $\Hom_{\dD}(P_i, Q_j[l])$, for $l\neq 0$, and the vanishing of $\Hom_{\dD}(P_i, E_n[l])$, for $l\neq 0$, proved above. Hence, $P$ is tilting.
	 

	Next we check that  $P$ is a classical generator for $\dD$. To this end we show 
	that $E_1,\ldots, E_n$ are objects of the subcategory $\langle P \rangle$ generated by $P$.   In view of \eqref{eqtn_def_P_i} objects $Q_1, \ldots, Q_{n-1}$ are in $\langle P \rangle$, i.e. $\langle Q \rangle \subset \langle P \rangle$. By inductive hypothesis, $E_1,\ldots, E_{n-1} \in \langle Q \rangle$. As $P_n = E_n$, we conclude that $E_1,\ldots, E_n \in \langle P \rangle$.

	
	Finally we check that $P$ is projective generator for $\aA$. As it is a tilting generator for $\dD$, category $\dD$ admits a \tr e with  
	\begin{align*}
	\dD^{\gge 0}_{P} = \{D\in \dD \,|\, \Hom_{\dD}(P, D[l])=0,\textrm{ for }l<0\}.
	\end{align*}
	The long exact sequence obtained by applying $\Hom_{\dD}(-,D)$ to \eqref{eqtn_def_P_i} and the isomorphism $P_n \simeq E_n$ imply that  if $D\in \dD^{\geq 0}_P$ then $\Hom_{\dD}(Q, D[l])$ vanish for $l<0$. On the other hand, the same sequence implies that, if $\Hom_{\dD}(Q, D[l])$ and $\Hom_{\dD}(P_n, D[l]) =0$, for $l<0$, then $D$ is an object of $\dD^{\geq 0 }_P$. Hence, as $E_n = {j_n}_!\kk$ and $Q = {i_n}_* Q$, 
	\begin{align*}
	\dD^{\gge 0}_P &= \{D\in \dD\,|\, \Hom_{\dD}({j_n}_!\kk, D[l]) = 0 = \Hom_{\dD}({i_n}_*Q, D[l]),  \textrm{ for }l<0\}\\
	& = \{D\in \dD\,|\, \Hom_{\kk}(\kk, j_n^*D[l]) = 0 = \Hom_{\dD_{n-1}}(Q, i_n^!D[l]),  \textrm{ for }l<0\}\\
	& = \{D\in \dD\,|\, j_n^*\in \dD^b(\kk)^{\gge 0}, i_n^!D\in {\dD_{n-1}}_Q^{\gge 0}\}.
	\end{align*}
	By inductive hypothesis the \tr e on $\dD_{n-1}$ given by $Q$ is the \tr e glued along the exceptional sequence $( E_1,\ldots, E_{n-1})$. Hence, in view of the definition of the glued \tr e \cite{BBD}, $\dD^{\gge 0}_P$ coincides with the positive aisle of the glued \tr e, i.e. these \tr es agree. 
	In particular, $P$ is a projective generator for the heart $\aA$.	
	
	It follows that $\aA \cong \textrm{mod-}\End_{\aA}P$ is a highest weight category with standard objects $\{E_i\}$, see Theorem \ref{thm_DR}. 
	By the Hom-duality \eqref{eqtn_dual_collec}, the left dual of the full exceptional sequence $(F_1,\ldots,F_n)$ in $\dD^{\opp}$ is $(E_n,\ldots, E_1)$. Moreover, this pair of dual exceptional sequences satisfies conditions in $(1)$. Hence,  the heart $\aA^{\opp}$ of the glued \tr e is highest weight with standard objects $\{F_i\}$ (see Corollary \ref{cor_dual_heart}), which implies that $\{F_i\}$ is the sequence of costandard objects in $\aA$, see Section \ref{ssec_hwc}. 
	
	Finally, let $\wt{\dD}$ be a DG enhancement for $\dD$ and $\wt{\aA}$ the full subcategory of $\wt{\dD}$ with objects $P_1, \ldots, P_n$. As $\dD = H^0(\wt{\dD})$ is idempotent complete and classically  generated by objects of $\wt{\aA}$, by \cite[Proposition 1.16]{LunOrl}, the category $\dD$ is equivalent to the category $\textrm{Perf}(\wt{\aA})$ of perfect complexes over $\wt{\aA}$. Further, since $P$ is tilting, the category $\textrm{Perf}(\wt{\aA})$ is equivalent to the category $\textrm{Perf}(\End_{\aA}P)$  of perfect complexes over the zero cohomology of $\wt{\aA}$. Finally, as $\aA$ is a highest weight category, it is of finite global dimension \cite{CPS}, hence $\textrm{Perf}(\End_{\aA} P) \simeq \dD^b(\aA)$.
\end{proof}

\vspace{0.3cm}
\subsection{Ringel duality in terms of dual pairs of exceptional sequences}~\\\label{Ring-dual}

Recall Ringel duality on highest weight categories \cite{Rin}. The Ringel dual $\mathbf{RD}(\aA)$ of a highest weight category $\aA$ is the abelian category of modules over the algebra of endomorphisms of the so-called {\it characteristic tilting module}. It is the additive generator of the category of objects which admit filtrations with both standard and costandard objects in $\aA$. In \cite{BodBon4}, we show that the characteristic tilting module is identified with an injective generator in the exact category $\mathcal{F}(\Delta)$ of objects admitting a filtration with quotients by direct sums of standard objects.

\begin{THM}\label{thm_Ringel} Let $\aA$ be a highest weight category. We assume it to be the heart of the \tr e glued along a full exceptional sequence $\epsilon=(E_1,\dots . E_n)$ in a triangulated category  $\dD$.
Then the Ringel dual $\mathbf{RD}(\aA)$ of $\aA$ is equivalent to the heart of the \tr e glued along the left dual exceptional sequence $( F_n, \ldots, F_1 )$.
\end{THM}
\begin{proof}
	Let $\mathbb{S}$ be the Serre functor on $\dD$, see \cite{B}, \cite{BK1}. Then $\mathbb{S}(\epsilon) = ( \mathbb{S}(E_1), \ldots, \mathbb{S}(E_n) )$ is left dual to $( F_n,\ldots, F_1 )$.  Since the heart $\aA$ is highest weight, negative Ext-groups between $E_i$'s and between  $F_i$'s vanish, according to Theorem \ref{thm_univ_ext_tilt_gen}. Moreover, $\{E_i\}$ are the standard objects in $\aA$ and $\{F_i\}$ are the costandard ones.
	
	Functor $\mathbb{S}$ is an equivalence, hence negative Ext-groups between  $\mathbb{S}(E_i)$'s vanish too. Then, Theorem \ref{thm_univ_ext_tilt_gen} implies that the heart $\bB$ glued along $(F_n,\dots , F_1)$ is highest weight with standard objects $\{F_i\}$ and costandard objects $\{\mathbb{S}(E_i)\}$.
	
	By \cite[Theorem 6.3]{BodBon4} $\mathbf{RD}(\aA)$ is highest weight with standard objects $\{\mathbb{S}^{-1} F_i\}$ and costandard objects $\{E_i\}$.
	By applying the Serre functor $\mathbb{S}$ we identify $\mathbf{RD}(\aA)$ with $\bB$.
\end{proof}

\begin{REM}\label{involut}
Note that the double dual of an exceptional sequence $\epsilon$ leads to the sequence $\mathbb{S}\epsilon$. Since $\mathbb{S}$ is an auto-equivalence, Theorem \ref{thm_Ringel} implies that Ringel duality is an involutive operation on equivalence classes of highest weight categories.
\end{REM}

\vspace{0.3cm}
\subsection{Restricted \tr e with the highest weight heart}~\\

Now we consider an arbitrary $t$-category and a (not necessarily full) exceptional sequence in it. We will need this set-up in section \ref{birat-surf} where we study the null-categories of birational morphisms of smooth surfaces.

\begin{THM}\label{thm_general}
	Consider a $t$-category $\dD$ with $\aA \subset \dD$ the heart of the \tr e. Let $\epsilon =(E_1,\ldots, E_n)$ be an exceptional sequence in $\dD$ with left dual $(F_n,\ldots, F_1)$. 
	\begin{enumerate}
		\item If, for any $i$, $E_i\in \dD^{\leq 0}$ and $F_i \in \dD^{\geq 0}$, then the \tr e on $\dD$ restricts to the  	\tr e on $\cC:= \langle E_1,\ldots,E_n\rangle$ glued along the exceptional sequence $(E_1,\ldots, E_n)$. The heart of this \tr e is $\aA \cap \cC$. 
		\item If $E_i, F_i\in \aA$, for all $i$, then $\aA\cap \cC$
		has a structure of the highest weight category with standard objects $\{E_i\}$, costandard $\{F_i\}$ and the partial order on irreducible objects induced by $\epsilon$.
	\end{enumerate}
\end{THM}
\begin{proof}
	Note that the \tr e on $\dD$ restricts to the glued \tr e on $\cC$ if and only if the embedding functor $\Phi \colon \cC\to \dD$ is $t$-exact \cite[1.3.19]{BBD}.
	
	We prove the $t$-exactness of the functor $\Phi$ by induction on length $n$ of the exceptional sequence. If $n=1$, the statement is clear. Consider recollement (\ref{eqtn_recol_l}) for $s=n$.  By inductive assumption $\Phi\circ {i_n}_*$ is $t$-exact. Since $\Phi j_{n!}(\kk) = E_n \in \dD^{\leq 0}$, $\Phi j_{n*}(\kk) = F_n \in\dD^{\geq 0}$, functor $\Phi\circ j_{n!}$ is right $t$-exact and $\Phi\circ j_{n*}$ is left $t$-exact. By Lemma \ref{lem_left_t-exact} below, functor $\Phi$ is $t$-exact, which proves (1).
	
	If $\{E_i\}$ and $\{F_i\}$ are in $\aA$, then $\Hom_{\dD}(\bigoplus E_i, \bigoplus E_i[l])$ and $\Hom_{\dD}(\bigoplus F_i, \bigoplus F_i[l])$ vanish for $l<0$. Hence, (2) follows from Theorem \ref{thm_univ_ext_tilt_gen}. 
\end{proof}

\begin{LEM}\label{lem_left_t-exact}
	Consider recollement
	\begin{equation}\label{eqtn_gen_recol}
	\xymatrix{\dD_1 \ar[r]|{i_*} & \dD \ar[r]|{j^*} \ar@<2ex>[l]|{i^!} \ar@<-2ex>[l]|{i^*} & \dD_2 \ar@<2ex>[l]|{j_*} \ar@<-2ex>[l]|{j_!}}
	\end{equation}
	and a $t$-category $\dD'$. Given \tr es on $\dD_1$ and $\dD_2$, an exact functor $\f\colon \dD \to \dD'$ is 
	\begin{enumerate}
		\item left $t$-exact for the glued \tr e on $\dD$ if $\f i_*$, $\f j_*$ are left $t$-exact.
		\item right $t$-exact for the glued \tr e on $\dD$ if $\f i_*$, $\f j_!$ are right $t$-exact.
	\end{enumerate}
\end{LEM}
\begin{proof}
	To prove (1), we need to show that $\f(\dD^{\gge 0}) \subset \dD'^{\gge 0}$. Any object $D\in \dD$ fits into an exact triangle
	$$
	i_*i^! D\to D \to  j_*j^*D\to i_*i^!D[1].
	$$
	If $D\in \dD^{\geq 0}$ then $i^!D$ lies in $\dD_1^{\gge 0}$ and $j^*D\in \dD_2^{\gge 0}$. The left $t$-exactness of $\f i_*$ and $\f j_*$ implies that
	$\f i_*i^!(D), \f j_*j^*(D) \in \dD'^{\gge 0}$, hence also $\f (D )\in \dD'^{\gge 0}$. 
	The proof of (2) is analogous.
\end{proof} 

\begin{COR}\label{cor_hwc_is_glued}
	Let $(\aA, \Lambda)$ be a highest weight category and $(\Delta(\la_1), \ldots, \Delta(\la_n))$, $(\nabla (\la_n), \ldots, \nabla(\la_1))$ the exceptional sequences of standard  and costandard objects. Then 
	\begin{enumerate}
		\item the full subcategory $\dD_i = \langle \Delta(\la_1), \ldots, \Delta(\la_i) \rangle$ and $ \langle \nabla(\la_i), \ldots, \nabla(\la_1) \rangle$ coincide.
		\item the standard \tr e on $\dD^b(\aA)$ is glued along $(\Delta(\la_1), \ldots, \Delta(\la_n))$.
	\end{enumerate}
%
\end{COR}
\begin{proof}
The first statement follows from the fact that $(\nabla (\la_n), \ldots, \nabla(\la_1))$ is left dual to $(\Delta(\la_1), \ldots, \Delta(\la_n))$, see Lemma \ref{lem_two_exc_coll}.

	By Lemma \ref{lem_two_exc_coll} the sequence of standard objects is full in $\dD^b(\aA)$.  Let $\aA'$ be the heart of the \tr e on $\dD^b(\aA)$ glued along $(\Delta(\la_1), \ldots, \Delta(\la_n))$. Theorem \ref{thm_general} applied to the standard \tr e on $\dD^b(\aA)$ yields $\aA' \simeq \aA \cap \dD^b(\aA)\simeq \aA$. The statement $(ii)$ follows.
\end{proof}
\begin{REM}\label{rem_dual_filt}
In view of Theorem \ref{thm_Ringel} and Corollary \ref{cor_hwc_is_glued}, given a highest weight category $\aA$ with standard objects $\{\Delta(\la_i)\}$, the Ringel dual \tr e on $\dD^b(\aA)$ is glued along the filtration
$$
\langle \Delta(\la_n) \rangle \subset \langle \Delta(\la_{n-1}), \Delta(\la_n) \rangle \subset \ldots \subset \langle \Delta(\la_2), \ldots, \Delta(\la_n)\rangle \subset \dD^b(\aA).
$$
\end{REM}

\vspace{0.3cm}
\subsection{Equivalence of highest weight structures via dual pairs of exceptional sequences}~\\

We discuss equivalence classes of highest weight structures on an abelian category $\aA$.

The example of M. Kalck \cite[Proposition 5.6]{Kalck} shows that the existence of a full exceptional sequence in $\dD^b(\aA)$ of objects in $\aA$ does not imply that $\aA$ is highest weight. However, we have:  

\begin{THM}\label{thm_main}
	Let $\aA$ be an abelian category. Then equivalence classes of highest weight structures on $\aA$ are in bijection with equivalence classes of full exceptional sequences in $\dD^b(\aA)$ such that both the sequence and its left dual one consist of objects in $\aA$.
\end{THM} 
\begin{proof}
	Lemma \ref{lem_two_exc_coll} gives a map in one direction. Theorem \ref{thm_general}(2) for $\dD= \dD(\aA)$ gives its inverse.
\end{proof}

\begin{REM}
For an arbitrary field $\kk$, an object $E$ in a $\kk$-linear Hom-finite triangulated category is \emph{exceptional} if $\Hom(E,E)=\Gamma$ is a division algebra and $\Hom(E,E[i])=0$, for $i\neq 0$. 
The notion of an exceptional sequence generalises accordingly, i.e. $(E_1, \ldots, E_n)$ is an exceptional sequence if $E_i$'s are exceptional and $\Hom(E_j,E_i[l])=0$, for all $l\in \mathbb{Z}$ and $j>i$. Under this more general definition, all 
statements of Section \ref{sec_hwc_via_ex_col} remain true. 
\end{REM}

\section{Algebraic examples}

Here we discuss in detail the example of M. Kalck \cite{Kalck} of an algebra which is not quasi-hereditary but admits a full exceptional sequence of modules. Then, we apply Theorem \ref{thm_general} to the description of two highest weight structures on the category of modules over a directed algebra. 

\vspace{0.3cm}
\subsection{Kalck's counterexample}~\\\label{ssec_Kalck}

We recall the example of M. Kalck \cite{Kalck} of an algebra $A$ which is not quasi-hereditary and which admits a full exceptional sequence of modules. We check that the standard \tr e on $\dD^b(A)$ is glued along the exceptional sequence and that the left and the right dual sequences do not consist of pure $A$-modules.

Algebra $A$ is the path algebra of a quiver
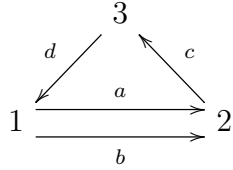
\[
\xymatrix{& 3 \ar[dl]_{d}& \\ 1 \ar@<1ex>[rr]^{a} \ar@<-1ex>[rr]_{b} & & 2 \ar[ul]_{c} }
\]
with relations
\begin{align*}
&a d=0,& &cb= 0,& &dc = 0.&
\end{align*}
By \cite[Lemma 5.5]{Kalck}, algebra $A$ does not admit a quasi-hereditary structure.

We denote by $S_1$, $S_2$, $S_3$ simple left $A$-modules and $P_1$, $P_2$, $P_3$ their irreducible projective covers. By \cite[Lemma 5.3]{Kalck} sequence 
$$
\sigma  = \langle S_3, P_2, P_1 \rangle
$$ 
is exceptional and full in $\dD^b(A)$. Its DG endomorphism algebra is the path algebra of  
\[
\xymatrix{S_3 \ar@<-2ex>[r]|{\epsilon} \ar[r]|{\f} & P_2 \ar@<-2ex>[r]|{\alpha} \ar[r]|\beta & P_1}
\]
with $\alpha$, $\beta$, $\epsilon$ of degree zero, $\f$ 
of degree two and
\begin{align}\label{eqtn_relations}
&\beta \epsilon = 0,& &\alpha\f =0.&
\end{align}

\begin{LEM}
	Left and right dual sequences to $\sigma$ do not consist of $A$-modules.
\end{LEM}
\begin{proof}
	We check that $L_{S_3}P_2$ and $R_{P_1}P_2$ are not pure.
	
	Equality $\Hom_A^\bcdot(S_3, P_2) =  \kk \oplus \kk[-2]$ implies that the left mutation of $P_2$ over $S_3$ fits into a distinguished triangle
	\begin{equation}\label{eqtn_L_S_3_P_2}
	S_3 \oplus S_3[-2] \xrightarrow{\left(\begin{array}{cc}\epsilon & \f 
		\end{array}\right)} P_2 \rightarrow L_{S_3}P_2\rightarrow S_3[1]\oplus S_3[-1].
	\end{equation}
	It follows that $H^0(L_{S_3}P_2)\simeq S_2$, $H^1(L_{S_3}(P_2)) \simeq S_3$.
	
	The space $\Hom_A(P_2, P_1)$ is two dimensional, hence the right mutation of $P_2$ over $P_1$ fits into a distinguished triangle
	\begin{align*}
	&R_{P_1}P_2 \to P_2 \to P_1^{\oplus 2} \to R_{P_1}P_2[1].&
	\end{align*}
	As the map $P_2 \to P_1^{\oplus 2} $ is not an epimorphism, $H^1(R_{P_1}P_2) \neq 0$.
\end{proof}

\begin{PROP}
	The standard \tr e on $\dD^b(A)$ is glued along filtration 
	$$
	\langle S_3 \rangle  \subset \langle S_3, P_2 \rangle \subset \langle S_3, P_2, P_1 \rangle
	$$
	from the standard \tr es on $\dD^b(k)$.
\end{PROP}
\begin{proof}
	We check that the heart $A\textrm{-mod}$ of the standard \tr e $(\dD^b(A)^{\leq 0}, \dD^b(A)^{\geq 1})$ is contained in the heart of the \tr e $(\dD^b(A)^{\leq_g 0}, \dD^b(A)^{\geq_g 1})$  glued along the filtration
	$$
	\langle S_3\rangle \subset \langle S_3, P_2 \rangle \subset \dD^b(A).
	$$
	It then follows that $\dD^b(A)^{\leq 0} = \dD^b(A)^{\leq_g0}$, i.e. the \tr es coincide.  Indeed, if $A\textrm{-mod}\subset \dD^b(A)^{\leq_g0}$, then also any object of $\dD^b(A)^{\leq 0}$, admitting a `filtration' with objects of $A\textrm{-mod}[l]$ for $l\geq 0$, belongs to $\dD^b(A)^{\leq_g 0}$. Analogously, $A\textrm{-mod}\subset \dD^b(A)^{\geq_g 0}$ implies that $\dD^b(A)^{\geq 1} \subset \dD^b(A)^{\geq_g 1}$. Hence, we have an inverse inclusion of the right orthogonal subcategories $\dD^b(A)^{\leq_g 0} \subset \dD^b(A)^{\leq 0}$. 
	
	We use the description of the aisles of the glued \tr e given by Lemma \ref{lem_glued_t-str} and check that the simple $A$-modules $S_1$, $S_2$, $S_3$ lie in their intersection. Then we conlcude that $A\textrm{-mod}$, i.e. the smallest subcategory of $\dD^b(A)$ containing $S_1$, $S_2$ and $S_3$, and closed under extensions, also belongs to $\dD^b(A)^{\leq_g0} \cap \dD^b(A)^{\geq_g 0}$ as needed.
	
	As $S_3$, $P_2$ and $P_1$ are $A$-modules, clearly $S_i \in \dD^b(A)^{\geq_g 0}$, for any $i$. 
	
The sequence left dual to $\sigma$ reads
\begin{equation} \label{eqtn_left_sigma}
\langle M, L_{S_3}P_2, S_3 \rangle,
\end{equation}
where $L_{S_3}P_2$ is defined by triangle \eqref{eqtn_L_S_3_P_2}. Object $M$ is defined in two steps; as $\Hom_A(P_2,P_1) = \kk^{2}$, we have an exact triangle
\begin{equation}\label{eqtn_L_P_2P_1}
	P_2 \oplus P_2 \xrightarrow{\left(\begin{array} {cc}\alpha & \beta \end{array}\right)} P_1 \to L_{P_2}P_1 \to P_2[1] \oplus P_2[1].
\end{equation}
Applying $\Hom_A(S_3,-)$ to it and using relations \eqref{eqtn_relations}, one gets that $\Hom_A(S_3, L_{P_2}P_1) = \kk[-1] \oplus \kk[1]$. Hence, object $M$ fits into an exact triangle
\begin{equation}\label{eqtn_M}
	S_3[-1] \oplus S_3[1] \to L_{P_2}P_1 \to M \to S_3 \oplus S_3[2].
\end{equation}
As the collection \eqref{eqtn_left_sigma} is exceptional, we have $S_3\in \dD^b(A)^{\leq_g 0}$.

To calculate morphisms from $S_1$ and $S_2$ to the objects of \eqref{eqtn_left_sigma}, we use their projective resolutions
\begin{align*}
	0 \to P_2 \to P_1 \to P_3\to P_2 \oplus P_2 \to P_1 \to S_1 \to 0,&\\
	0 \to P_2 \to P_1 \to P_3 \to P_2 \to S_2 \to 0.&
\end{align*}
 The resolution of $S_1$ yields $\Hom_A^{\bcdot}(S_1, S_3 )=\kk[-2]$ and $\Hom_A^{\bcdot}(S_1, P_2) \in \dD^b(\kk)^{\geq 1}$. Hence, triangle \eqref{eqtn_L_S_3_P_2} implies that $\Hom_A^{\bcdot}(S_1, L_{S_3}P_2) \in \dD^b(\kk)^{\geq 1}$. Further, $\Hom_{A}^{\bcdot}(S_1, P_1)\in \dD^b(\kk)^{\geq 1}$, hence $\Hom_A^{\bcdot}(S_1, L_{P_2}P_1) \in \dD^b(\kk)^{\geq 0}$. Applying $\Hom_A(S_1,-)$ to \eqref{eqtn_M} one gets that $\Hom_A^{\bcdot}(S_1,M) \in \dD^b(\kk)^{\geq 0}$. It follows that $S_1\in \dD^b(A)^{\leq_g 0}$.
 
 Analogously, the projective resolution of $S_2$ implies that $\Hom_A^{\bcdot}(S_2,S_3) = \kk[-1]$ and $\Hom_A^{\bcdot}(S_2,P_2) \in \dD^b(\kk)^{\geq 3}$. It follows that $\Hom_A^{\bcdot}(S_2, L_{S_3}P_2) \in \dD^b(\kk)^{\geq 0}$.
 
 To show that negative Ext-groups from $S_2$ to $M$ vanish we first note that triangle \eqref{eqtn_L_P_2P_1} implies that object $L_{P_2}P_1$ has two non-zero cohomology objects and fits into an exact triangle
 \begin{equation}\label{eqtn_coh_L_P_2P_1}
 	S_3[1] \xrightarrow{\iota} L_{P_2}P_1 \to S_1 \to S_3[2].
 \end{equation}
This triangle implies in particular, that $\Hom_A(S_3[1], L_{P_2}P_1)$ is spanned by the truncation $\iota \colon \tau^{\leq -1}L_{P_2}P_1\to L_{P_2}P_1$. By applying $\Hom_A(S_2, -)$ to \eqref{eqtn_coh_L_P_2P_1} we see that the morphism $\Hom_A(S_2, S_3[1]) \to \Hom_A(S_2, L_{P_2}P_1)$, induced by $\iota$, is an isomorphism. As $\iota$ is the component $S_3[1] \to L_{P_2}P_1$ in \eqref{eqtn_M}, the map $\Hom_A(S_2, S_3[-1] \oplus S_3[1]) \to \Hom_A(S_2, L_{P_2}P_1)$ in the long exact sequence obtained by applying $\Hom_A(S_2, -)$ to \eqref{eqtn_M}, is an isomorphism. It follows that $\Hom_A^{\bcdot}(S_2, M) \in \dD^b(\kk) ^{\geq 0}$, hence $S_2 \in \dD^b(A)^{\leq_g 0}$. 
\end{proof}

As the standard \tr e on $\dD^b(A)$ is glued, the category of modules over $A$ is a quasi-hereditary category in the sense of \cite{Bez1}. The standard objects are elements $S_3$, $P_2$, $P_1$ of $\sigma$, the costandard are $H^0(M)$, $H^0(L_{S_3}P_2)$ and $S_3$. Sequence \eqref{eqtn_M} implies that  $H^0(M)$ is the unique extension of $S_3$ by $S_1$. As discussed above,  $H^0(L_{S_3}P_2) \simeq S_2$. As $\Ext_A^2(S_3,S_2) =k$, the sequence 
$$
\tau =\langle H^0(M), S_2,S_3 \rangle
$$ 
is not exceptional.

\vspace{0.3cm}
\subsection{Directed algebras}~\\ \label{ssec_dir_alg}

Let $A$ be a path algebra of a directed quiver  with relations and $\aA$ the category of left $A$-modules. We assume that vertices of the quiver are labelled with $\{1,\ldots,n\}$ and for every arrow the index of the source is less than the index of the target.

Indecomposable projective objects $P_i$ in $\aA$ correspond to vertices $\{1,\ldots,n\}$ of the quiver. As the quiver is directed, $\Hom(P_i,P_j) = 0$ for $i>j$ and $\Hom(P_i,P_i)=\kk$. Hence, $(P_n,\ldots,P_1)$ is a full exceptional sequence in $\dD^b(\aA)$. 

Let $S_i$ be the simple $A$-module corresponding to the vertex $i$. As $\dim_{\kk} \Hom(P_i,S_j) = \delta_{ij}$, $(S_1,\ldots,S_n)$ is a full exceptional sequence left dual to $(P_n,\ldots,P_1)$. 

Finally, let $I_i$ denote the injective envelope of $S_i$, i.e. the indecomposable injective $A$-module corresponding to the vertex $i$. As $\dim \Hom(S_i,I_j) = \delta_{ij}$, $(I_n,\ldots,I_1)$ is a full exceptional sequence in $\dD^b(\aA)$ left dual to $(S_1,\ldots, S_n)$.

Theorem \ref{thm_general} implies that $\aA$ admits two structures of a highest weight category. In the first one $P_i$ are the  standard objects while $S_i$ are costandard. The induced partial order on the set $\{1,\ldots,n\}$ is inverse to the standard one. Equivalently, one can consider simple $A$-modules as standard objects and indecomposable injective objects as costandard. Then the corresponding partial order on the set $\{1,\ldots,n\}$ is the standard one.

\vspace{0.3cm}
\section{Geometric applications}

In this section, we apply Theorem \ref{thm_general} to proving highest weight structure on the category of perverse sheaves with the middle perversity on a complex-analytic space stratified by consecutive divisors with contractible open parts. 

Next, we consider the derived category of a Grassmannian in arbitrary characteristic and give  an alternative proof for the existence of a \tr e with the  heart admitting a highest weight structure \cite{BuchLeuVdB} \label{thm_Gr}\cite{Efi}. 

Finally, we prove that the abelian null category for any birational morphism of regular surfaces is highest weight. We provide an explicit geometric description for  standard, costandard and characteristic tilting objects. Also we show that the triangulated null category is equivalent to the derived category of the abelian null category in this case.

\vspace{0.3cm}
\subsection{Perverse sheaves on stratified complex analytic spaces}~\\
\label{perverses}
The theory of glued \tr es captures abstract algebraic properties of the theory of perverse sheaves. Let us illustrate how our approach is applicable to this classical geometric situation.

Let $X$ be a complex analytic space with an analytic stratification $\mathscr{S}$ satisfying the Whitney condition. We endow the set $\mathscr{S}$ of strata with the partial order, for which $S_1\leq S_2$ if and only if $S_1 $  is contained in the closure $ \overline{S_2}$ of $S_2$.

For a strata $S$, denote by $\dD(S)$ the derived category of sheaves with locally constant cohomology and by $j_S\colon S\to X$ the inclusion. 
Consider the derived category $\dD_{\mathscr{S}}(X)$ of sheaves on $X$ with $\mathscr{S}$-constructible cohomology. Let $U\subset X$ be an open strata. Denote by $\mathscr{S}'$ the induced stratification on the complex analytic space $X\setminus U$. Then $\dD_{\mathscr{S}}(X)$ fits into a recollement \cite{BBD}
\begin{equation}\label{eqtn_per_rec}
\xymatrix{\dD_{\mathscr{S}'}(X\setminus U) \ar[r]& \dD_{\mathscr{S}}(X) \ar[r] \ar@<2ex>[l]  \ar@<-2ex>[l] & \dD(U).\ar@<2ex>[l] \ar@<-2ex>[l]}
\end{equation}
Consider the recollements of the form \eqref{eqtn_per_rec} for complex analytic spaces obtained by consecutively removing open strata from $X$.

Consider the \tr e on $\dD_{\mathscr{S}}(X)$  which is recursively glued via the above recollements  from the standard \tr es on $\dD(S)$ shifted by $-\dim_{\mathbb{C}}(S)$ \cite{BBD}. Denote by $\textrm{Per}_{\mathscr{S}}(X)$ its heart, called the category of \emph{perverse sheaves with  middle perversity}  on $X$.

We further assume that 
\begin{itemize}
	\item[(S1)] all strata  $S\in \mathscr{S}$ are contractible and
	\item[(S2)] $\ol{S} \setminus S$ is either empty or a Cartier divisor in $\ol{S}$.
\end{itemize}
Here, by a Cartier divisor in a complex analytic space we mean a closed subset locally given as zeroes of a holomorphic function. 

For a contractible stratum $S$, we have $\dD(S) \simeq \dD^b(\kk)$. Then the only irreducible local system on $S$ is just the constant sheaf $\kk_S$. 
If condition (S1) is satisfied, the irreducible objects in $\textrm{Per}_{\mathscr{S}}(X)$ are intermediate extensions ${j_S}_{!*}\kk_S[-\dim_{\mathbb{C}} S] := \textrm{image}({}^p{j_S}_{!}\kk_S[-\dim_{\mathbb{C}} S]\to {}^p{j_S}_{*}\kk_S[-\dim_{\mathbb{C}} S])$ \cite{BBD}.



Our main theorem gives a different proof of the well-known fact that $\textrm{Per}_{\mathscr{S}}(X)$ is a highest weight category.

\begin{PROP}\cite{MirVil, ParSco}
	Let $X$ be complex analytic space and $\mathscr{S}$  a stratification satisfying (S1) and (S2). Consider the partial order $\Lambda$ on irreducible objects in  $\textrm{Per}_{\mathscr{S}}(X)$  induced by the partial order on  $\mathscr{S}$. Then $(\textrm{Per}_{\mathscr{S}}(X), \Lambda)$ 
	is a highest weight category. Moreover, $\dD^b(\textrm{Per}_{\mathscr{S}}(X)) \cong \dD_{\mathscr{S}}(X)$.
\end{PROP}
\begin{proof}
	Consider recollements of the form \eqref{eqtn_per_rec}.
	One proves by induction on the number of strata  that $\langle {j_S}_! \kk_S [-\dim_{\mathbb{C}}S] \rangle$ is a full exceptional sequence in $\dD_{\mathscr{S}}(X)$ whose left dual sequence is $\langle {j_S}_* \kk_S [-\dim_{\mathbb{C}}S] \rangle$ . 
	
	Since $\ol{S}\setminus S$ is Cartier, functor ${j_S}_!$ is $t$-exact for the shifted \tr e on $\dD(S)$ and the glued \tr e on $\dD_{\mathscr{S}}(X)$ \cite[Lemma 3.1]{MirVil}. As the \tr e is invariant under Verdier duality $\mathscr{D}_X\colon \dD_{\mathscr{S}}(X)^{\opp}\to \dD_{\mathscr{S}}(X)$ and $\mathscr{D}_X \circ {j_S}_!\cong {j_S}_*\circ \mathscr{D}_X$, functor ${j_S}_*$ is also $t$-exact for the same \tr es. 
	Hence, both sequences  $\langle {j_S}_!\kk_{S}[-\dim_{\mathbb{C}}S]\rangle$ and $\langle {j_S}_*\kk_{S}[-\dim_{\mathbb{C}}S]\rangle$ lie in $\textrm{Per}_{\mathscr{S}}(X)$.
	 Then, by Theorem \ref{thm_general}(2) category $\textrm{Per}_{\mathscr{S}}(X)$ is highest weight. As $\dD_{\mathscr{S}}(X)$ admits a DG enhancement, Theorem \ref{thm_univ_ext_tilt_gen}(2) provides an equivalence $\dD^b(\textrm{Per}_{\mathscr{S}}(X)) \cong \dD_{\mathscr{S}}(X)$.
\end{proof}

\vspace{0.3cm}

\subsection{Grassamannians in finite characteristic}~\\
\label{secGrass}
Let $\mathbb{G} = {\bf G}(l,F) = {\bf G}(l,m)$ be the Grassmannian of $l$-dimensional subspaces of an $m$ dimensional $\kk$-vector space $F$. Consider the tautological sequence of vector bundles
\[
0 \to \rR \to F^{\vee} \otimes_{\kk} \oO_{\mathbb{G}} \to \qQ \to 0
\]
with $\rR$ of rank $l$ and $\qQ$ of rank $m-l$. For a partition $\alpha$, let $S^\alpha$ be the associated Schur functor. Denote by $\alpha^T$ the transpose of $\alpha$ and by $|\alpha|$ its degree. Let $B_{u,v}$ be the set of partitions with at most $u$ rows and at most $v$ columns equipped with a total ordering $\prec$ such that if $|\alpha|<|\beta|$ then $\alpha \prec \beta$. Finally let $\ol{B}_{u,v}$ be $B_{u,v}$ with the inverse total ordering.

\begin{THM}\cite{Kap1, BuchLeuVdB, Efi}
	The category $\dD^b(\mathbb{G})$ admits a full exceptional sequence $\langle S^\alpha \qQ \rangle_{\alpha \in B_{l,m-l}}$ whose left dual sequence is $\langle S^{\alpha^T} \rR[-|\alpha|] \rangle_{\alpha \in \ol{B}_{l,m-l}}$.
\end{THM}

The sequence was constructed by M. Kapranov \cite{Kap1} when the characteristic of $\kk$ is zero and by O. Buchweitz, G. Leuschke and M. Van den Bergh \cite{BuchLeuVdB} for an arbitrary field. A. Efimov \cite{Efi} generalised these results to the Grassmannian defined over $\mathbb{Z}$.

If ${\rm char}\kk =0$, the sequence $\langle S^\alpha \qQ \rangle_{\alpha \in B_{l,m-l}}$ is strong, in particular its endomorphism algebra is directed, hence quasi-hereditary, see Section \ref{ssec_dir_alg}. It follows that the highest weight category $\modu B$ is the heart of the glued \tr e in $\dD^b(\Coh(\mathbb{G}))$. When ${\rm char}\kk$ is small, the sequence is not necessarily strong. However, we have

\begin{THM}\cite{BuchLeuVdB} \label{thm_Gr}\cite{Efi}
	The heart $\aA$ of the \tr e glued along the exceptional sequence $\langle S^\alpha \qQ \rangle_{\alpha \in B_{l,m-l}}$ is highest weight and $\dD^b(\Coh(\mathbb{G})) \simeq \dD^b(\aA)$.
\end{THM}
\begin{proof}
As the category $\dD^b(\Coh(\mathbb{G}))$ is idempotent complete and DG enhanced it suffices to check that $\Hom_{\mathbb{G}}(\bigoplus S^\alpha \qQ, \bigoplus S^\alpha\qQ[l])$ and $\Hom_{\mathbb{G}}(\bigoplus S^{\alpha^T} \rR[-|\alpha|], \bigoplus S^{\alpha^T}\rR[l-|\alpha|])$ vanish for $l<0$. Indeed, if this is the case then the statement follows direclty from Theorem \ref{thm_univ_ext_tilt_gen}.

The vanishing of the first groups follows immediately from the fact that $S^\alpha \qQ$ are vector bundles on $\mathbb{G}$. Let us now consider the group $\Hom_{\mathbb{G}} (S^{\alpha^T} \rR[-|\alpha|], S^{\beta^T}\rR[l-|\beta|]) = \Hom_{\mathbb{G}}(S^{\alpha^T}\rR, S^{\beta^T}\rR[l+|\alpha| - |\beta|])$. If $|\alpha|-|\beta|\leq 0$ the group vanishes for any $l<0$, because $S^{\alpha^T} \rR$ and $S^{\beta^T} \rR$ are vector bundles on $\mathbb{G}$. If, on the other hand, $|\alpha| - |\beta|>0$, then the vanishing of the group for any $l$ follows from the relation $\beta \prec \alpha$ in $\ol{B}_{u,v}$.
\end{proof}

\begin{REM}
	In \cite{BuchLeuVdB} an explicit tilting object $T$ for $\dD^b(\Coh(\mathbb{G}))$ is constructed. The authors check that $\modu \End(T)$ is highest weight and it proved that  the equivalence $\Hom(T,-)\colon \dD^b(\Coh(\mathbb{G})) \to \dD^b(\modu \End(T))$ maps $S^\alpha \qQ$ to the standard objects. A different argument in \cite{Efi} proves that a statement analogous to Theorem \ref{thm_Gr} holds over the ring of integers.
\end{REM}

\begin{REM}

	A full exceptional sequence has also been constructed by \cite{KuzPol} on the Grassmannian ${\rm LGr}(n,V)$ of Lagrangian subspaces in a $2n$-dimensional vector space $V$ over a field $\kk$ of characteristic zero equipped with a non-degenerate symplectic form. The sequence consists of vector bundles, hence negative $\Ext$-groups between its elements vanish.  The sequence is built from $n+1$ strong blocks. Recently, the left dual sequences in these blocks were described by  A. Fonarev \cite{Fon2}. 
	It would be interesting to describe the left dual of the entire exceptional sequence and check if it satisfy the conditions of the criterion  in 
Theorem \ref{thm_univ_ext_tilt_gen} for the glued heart to be a highest weight category . 
\end{REM}

%

\vspace{0.3cm}
\subsection{Null categories for birational morphisms of regular surfaces}\label{birat-surf}~\\



Here we give an explicit description of an exceptional sequence and its dual in the null category for any birational morphism of regular surfaces.

Let $f\colon X\to Y$ be a  proper birational morphism of regular surfaces over $k$. We consider the \emph{triangulated null category} \cite{BodBon2}:
$$
\cC_f = \{E \in \dD^b(X)\,|\, Rf_*(E) =0\}.
$$

For $f$ as above we consider the partially ordered set $\Conn(f)$, see  \cite{BodBon2}. Elements of $\Conn(f)$ are pairs $(g,h)$ which provide decompositions $f = h \circ g$ of $f$ into proper birational morphisms $g$ and $h$ of regular surfaces such that the exceptional divisor of $g$ is connected. Given $(g_1,h_1)$, $(g_2, h_2)$ in $\Conn(f)$ we put $(g_1, h_1) \preceq (g_2, h_2)$ if $g_2$ factors via $g_1$.

Let $(g,h)$ be an element of $\Conn(f)$ with $g\colon X\to Z$. We denote by $R_g$ the fundamental cycle of $g$ and by $D_g$ its discrepancy divisor. If $z\in Z$ is the image of the exceptional divisor of $g$ and $\mathfrak{m}_g$ its maximal ideal sheaf, then $R_g$ and $D_g$ are divisors in $X$ defined by the equalities: 
\begin{align*}
&g^{-1}(\mathfrak{m}_g) = \oO_X(-R_g),& & g^!(\oO_Z) = \oO_X(D_g).&
\end{align*}

We denote by $\Conn(f)^{\opp}$ the opposite partial order on the set $\Conn(f)$.


\begin{THM}\label{thm_exc_coll_C_f}
	Category $\cC_f$ admits full exceptional sequence $\left(\oO_{R_g}(R_g)\right)_{g\in \Conn(f)}$ whose left dual sequence is $\left( \oO_{R_g}(D_g) \right)_{g\in \Conn(f)^{\opp}}$.
\end{THM}

\begin{proof}
We proceed by induction on the number $n$ of irreducible components in the exceptional divisor for $f$. If $n=1$ then $f$ is the blow-up of a smooth point. Then $(f, \textrm{Id})$ is the single element of $\Conn(f)$ and divisors $R_f$ and $D_f$ coincide. If this is the case, then the null category is equivalent to $\langle \oO_{R_f}(R_f) \rangle = \langle \oO_{D_f}(R_f) \rangle$.

Let now the exceptional divisor $\textrm{Ex}(f)$ of $f\colon X\to Y$ have $n$ irreducible components. Without loss of generality we can assume that $\textrm{Ex}(f)$ connected. Indeed, if this is not the case, let $g_1 \colon X\to Z_1$ and $g_2 \colon X\to Z_2$ be birational proper morphisms to regular surfaces such that $\textrm{Ex}(f) = \textrm{Ex}(g_1) \sqcup \textrm{Ex}(g_2)$. Then $\cC_f = \cC_{g_1} \oplus \cC_{g_2}$ and $\Conn(f)$ is a union of $\Conn(g_1)$ and $\Conn(g_2)$ such that any element of $\Conn(g_1)$ is not comparable with any element of $\Conn(g_2)$.

Let $y \in Y$ be the image of $\textrm{Ex}(f)$ and let $h\colon Z\to Y$ be the blow-up at $y$. Finally, let $g\colon X \to Z$ be such that $f = h \circ g$. Then $\Conn(f) = \Conn(g) \cup \{f\}$ and category $\cC_f$ admits semi-orthogonal decompositions
\begin{equation}
\cC_f = \langle \cC_g, Lg^* \cC_h \rangle = \langle g^! \cC_h, \cC_g \rangle.
\end{equation}
As $h$ is the blow-up at a smooth point, category $\cC_h$ is generated by $\oO_E(E)$, for the exceptional divisor $E$ of $h$. Since $Lg^* \oO_E(E) = \oO_{R_f}(R_f)$ and $g^! \oO_E(E) = \oO_{R_f}(D_f)$ the inductive hypothesis imply that $\left( \oO_{R_g}(R_g)\right)_{g\in \Conn(f)}$ and $\left( \oO_{R_g}(D_g) \right)_{g\in \Conn(f)^{\opp}}$ are exceptional sequences in $\cC_f$. As $\Hom_X(\oO_{R_g}(R_g), \oO_{R_g}(D_g)) = \kk$ the second one is indeed left dual to the first one.
\end{proof}
\begin{REM}
	For the case of rational surfaces, the exceptional sequence $\left(\oO_{R_g}(R_g)\right)_{g\in \Conn(f)}$ may be interpreted as a mutation of a part of the exceptional sequence considered by Hille and Perling in \cite{HilPer2}.
\end{REM}


Now we introduce the \emph{abelian null category} for $f\colon X\to Y$:
$$
\mathscr{A}_f = \{E\in \Coh(X)\,|\, Rf_*(E) =0\}.
$$

The exceptional divisor $\textrm{Ex}(f) = \bigcup C_i$ of $f$ is a tree of rational curves (see \cite[Theorem D.1]{BodBon} for a more general statement). By \cite[Theorem 7.13]{BodBon6}, $\{\oO_{C_i}(-1)\}_{C_i \in \textrm{Irr}(f)}$ is the set of isomorphism classes of irreducible objects in $\mathscr{A}_f$. Moreover, by \cite[Theorem 2.17]{BodBon2} the set of irreducible components of $\textrm{Ex}(f)$ is in bijection with $\Conn(f)$. Hence, the partial order on $\Conn(f)$ induces a partial order on the set of irreducible objects in $\mathscr{A}_f$. We denote it by $\Lambda$.



\begin{THM}\label{thm_surfaces}
	Consider a proper birational morphism $f\colon X\to Y$ of regular surfaces. Then the standard \tr e on $\dD^b(X)$ restricts to a \tr e on $\cC_f$ with heart $\mathscr{A}_f $.  $(\mathscr{A}_f,\Lambda)$ is a highest weight category with standard objects $\{\oO_{R_g}(R_g)\}_{g\in \Conn(f)}$ and costandard ones $\{\oO_{R_g}(D_g)\}_{g\in \Conn(f)}$. Moreover, $\cC_f\cong \dD^b(\mathscr{A}_f)$.
\end{THM}
\begin{proof}
	By Theorem \ref{thm_exc_coll_C_f}, category $\cC_f$ admits a full exceptional  sequence  $\left( \oO_{R_g}(R_g)\right)_{g\in \Conn(f)}$ of sheaves whose left dual sequence $\left(  \oO_{R_g}(D_{g}) \right)_{g\in \Conn(f)^{\opp}}$ also consists of sheaves. Hence, by Theorem \ref{thm_general}, the standard \tr e on $\dD^b(X)$, with heart $\Coh(X)$, restricts to a \tr e on $\cC_f$ with heart $\mathscr{A}_f = \cC_f\cap \Coh(X)$. Further, $(\mathscr{A}_f,\Lambda)$ is highest weight with the above standard and costandard objects. Finally, as $\cC_f$ is idempotent complete and admits a DG enhancement, Theorem \ref{thm_univ_ext_tilt_gen}(2) implies that $\cC_f\cong\dD^b(\mathscr{A}_f)$.
\end{proof}
 This theorem provides an interesting class of highest weight categories of geometric origin which deserves further investigation.

Note that the fact that, for a morphism $f\colon X\to Y$ between schemes, the standard \tr e on $\dD^b(X)$ restricts to a \tr e on $\cC_f$ when fibers of $f$ have dimension $\leq 1$ is proved in \cite{Br1}. Hence, $\mathscr{A}_f$ is abelian in this case.


We do not have an explicit description of projective or injective generators for $\mathscr{A}_f$. However, the characteristic tilting object can be written down explicitly.
\begin{THM}\label{thm_tilting_for_surf}
The sheaf
$$
T_f = \bigoplus_{g\in \textrm{Conn}(f)} \omega_X|_{D_g}
$$
is the characteristic tilting module for the highest weight category $(\mathscr{A}_f, \Lambda)$.
\end{THM}

\begin{proof}
By \cite[Section 6.1]{BodBon5} the characteristic tilting object is the injective generator for the heart of the \tr e on $\cC_f$ glued along the exceptional sequence right dual to the sequence $\langle \oO_{R_g}(R_g) \rangle_{g\in \Conn(f)}$ of standard objects (see Theorem \ref{thm_surfaces}). 

It follows from Remark \ref{rem_dual_filt} and Theorem \ref{thm_surfaces}  that the Ringel dual \tr e on $\cC_f$ is glued along the $\Dec(f)^{\opp}$-filtration
$$
Lg^* \cC_h \subset \cC_f, 
$$
i.e. for any decomposition $f= h \circ g$ into proper birational morphisms of regular surfaces, the functor $Lg^* \colon \cC_h \to \cC_f$ is $t$-exact for the Ringel dual \tr e on $\cC_h$ and $\cC_f$, see \cite{BodBon2}.

%
%

Then it follows from \cite[Proposition 5.14]{BodBon2} by induction on the number of standard objects in $\cC_f$ that the Ringel dual \tr e on $\cC_f$ is the $T_f$-injective \tr e. The statement follows.  
\end{proof}

If $Y$ is a singular surface and $f\colon X\to Y$ its resolution then $\mathscr{A}_f$ is not necessarily a highest weight category and $\cC_f$ does not have to be equivalent to $\dD^b(\mathscr{A}_f)$ as the following examples show.

\begin{EXM}\label{typeA2}
Consider the minimal resolution $f\colon X\to Y$ of type $A_2$ surface singularity. The abelian null category $\mathscr{A}_f$ is equivalent to the category of finite dimensional modules over the quiver
\[
\xymatrix{ 1 \ar@<1ex>[r]^a & 2 \ar@<1ex>[l]^b}
\]
with relations $ab =0 $, $ba=0$. Hence, $\mathscr{A}_f$ is of infinite global dimension while any highest weight category is of finite global dimension \cite{CPS}.
Therefore, $\mathscr{A}_f$ is not highest weight for any partial order on the set of simple modules.
\end{EXM}
\begin{EXM}\label{typeA1}
Let now $f\colon X\to Y$ be the minimal resolution of type $A_1$ surface singularity. If $E\subset X$ is the exceptional divisor of $f$, then $\oO_E(-1)$ is an additive generator for the category $\mathscr{A}_f$, equivalent to the category of finite dimensional $k$-vector spaces. However, $\oO_E(-1)$ is a 2-spherical object. In particular,  $\Ext^2_X(\oO_E(-1), \oO_E(-1)) \cong \kk$. Hence,  $\dD^b(\mathscr{A}_f)$ is not equivalent to $\cC_f$. 
\end{EXM}
In higher dimension, consider a flopping contraction $f\colon X\to Y$ with following conditions: dimension of fibers is bounded by 1, the exceptional locus is of codimension in $X$ greater than 1, both $X$ and $Y$ Gorenstein, and $Y$ with canonical hypersurface singularities of multiplicity two. In this case also $\dD^b(\mathscr{A}_f)\ncong \cC_f$. Moreover, the derived functor $\dD^b(\mathscr{A}_f) \to \dD^b(X)$ of the inclusion $\mathscr{A}_f\subset \Coh(X)$ is proved to be spherical \cite{BodBon}.

	\bibliographystyle{alpha}
		\bibliography{../ref}

\end{document}